\documentclass{amsart}
\usepackage{amssymb}
\usepackage{url}
\usepackage{hyperref}
\usepackage{color}
\usepackage[usenames,dvipsnames]{xcolor}
\usepackage[english]{babel}
\def \Det {\mathrm{Det}}
\newcommand{\bel}{\begin{equation}\label }
\newcommand{\nobel}{\begin{equation}}
\newcommand{\ee}{\end{equation}}
      \newtheorem{theorem}{Theorem}[section]
       \newtheorem{proposition}[theorem]{Proposition}
       
       \newtheorem{lemma}[theorem]{Lemma}
       \newtheorem{remark}{Remark}[section]
\theoremstyle{definition}

\def \P {{\mathbb P}}

\def \E {{\mathbb E}}

\def \e {{\mathbf e}}

\def \R {{\mathbb R}}
\def \K {{\mathcal K}}
\def \mk {{\mathcal {M K}}}
\def \G {{\mathcal G}}
\def \W {{\mathcal W}}
\def\tr{\;\mathrm{tr}\,}
\def\j{\mathrm{J}}
\def\<{\langle}
\def\>{\rangle}
\begin{document}

\title[Independence characterization for Wishart and Kummer random matrices]{Independence characterization for Wishart and Kummer random matrices}
\authors{Bartosz Ko{\l}odziejek, Agnieszka Piliszek}
\address{Wydzia{\l} Matematyki i Nauk Informacyjnych\\
Politechnika Warszawska\\
Koszykowa 75\\
00-662  Warszawa, Poland}
\email{A.Piliszek@mini.pw.edu.pl}
\email{B.Kolodziejek@mini.pw.edu.pl}
\date{\today}

\begin{abstract}
We generalize the following univariate characterization of the Kummer and Gamma distributions to the cone of symmetric positive definite matrices: let $X$ and $Y$ be independent, non-degenerate random variables valued in $(0, \infty)$, then $U= Y/(1+X)$ and $V = X(1+U)$ are independent if and only if $X$  follows the Kummer distribution and $Y$ follows the the Gamma distribution with appropriate parameters. We solve a related functional equation in the cone of symmetric positive definite matrices, which is our first main result and apply its solution to prove the characterization of Wishart and matrix-Kummer distributions, which is our second main result. 
\end{abstract}

\maketitle

\section{Introduction}
\label{sec:1}
In 1940s Bernstein noticed that if $X$ and $Y$ are independent, then $X-Y$ and $X+Y$ are independent if and only if $X$ and $Y$ are Gaussian, \cite{Be41}. This observation suggested that independence can mean more than one could think. Many other examples of so-called independence characterizations have been identified through the years. One of the highlights in this area is Lukacs' (1955) characterization of the Gamma distribution by the independence of $X+Y$ and $X/(X+Y)$, \cite{Lu55}.
In 1996 Casalis and Letac wrote, that independence characterizations of distributions \textit{give insight into the laws of nature and may reveal quite beautiful mathematics},  \cite{CL96}. In the cited paper they showed a new way, compared to Olkin and Rubin, \cite{OR62}, to generalize the Lukacs theorem to symmetric positive definite matrices.

Another celebrated characterization origins from the Matsumoto-Yor (MY) property, see \cite{MY01}, \cite{MY03}, which says that for independent $X$ and $Y$ having GIG and Gamma distributions, random variables $1/(X+Y)$ and $1/X - 1/(X+Y)$ are also independent.  First characterization of GIG and Gamma distributions through this property was given in \cite{LW00}. It has been widely generalized and modified: symmetric cones \cite{Ko15_s}, free probability \cite{KSz17} and others \cite{BW03}, \cite{MW03}, \cite{MW06}. In 2012 a whole family of independence properties of a MY type  was given by Koudou and Vallois \cite{KV12}. The latter paper  presents all possible distributions of independent $X$ and $Y$ for which there exists a (very regular, see \cite{KV12}) function $f$ such that $f(X+Y)$ and $f(X) - f(X+Y)$ are also independent. The Lukacs property corresponds to  $f(x) = \log x$ and the MY property to $f(x) = 1/x$. Another important case identified in \cite{KV12} was $f(x)=\ln(1+ 1/x)$. That one concerns Kummer and Gamma distributions and can be formulated as follows: Let $X$ has the Kummer distribution $\K (a, b, c)$ with density 
\bel{kumer} f_X(x) \propto x^{a-1} (1+x)^{-(a+b)} e^{-cx} I_{(0,\infty)}(x)\ee
and $Y$ has the Gamma distribution  $\G(b,c)$ with density
$$ f_Y(y) \propto y^{b-1} e^{-cy} I_{(0, \infty)}(y),$$ 
where $a,b,c>0$. Suppose that $X$ and $Y$ are independent and let
\bel{KV1}
U=X+Y \quad \mbox{and} \quad V = \frac{1+{1}/(X+Y)}{1+{1}/{X}}
.\ee Then $U$ and $V$ are also independent. 

To derive related characterization, however, the authors  needed to impose technical conditions of differentiability (\cite{KV12}) or local integrability (\cite{KV11}) of logarithms of strictly positive densities.  Recently a regression version of this characterization under natural integrability assumption (and with no assumptions concerning densities) was given in \cite{We15}. In \cite{PWkiedys} even the integrability assumption was cleared out through the change of measure technique. In the last-mentioned paper also another independence property and a related characterization concerning Kummer and Gamma distributions were considered. The property was formulated by Hamza and Vallois in \cite{HV16} and we will call it HV property in the sequel.  It says that if  $X\sim\K(a, b-a, c)$ (which means that $X$ has Kummer distribution with parameters $a$, $b-a$ and $c$) and $Y\sim\G(b,c)$, $a,b,c>0$ are independent random variables  and  if
\bel{zero}
T_0(x,y)=\left(y/(1+x) ,\;x\,\left(1 +y/(1+x)\right)\right),
\ee 
then the
random vector $(U,V) = T_0(X,Y)$ has independent components, $U\sim\K(b,a -b, c)$ and $V\sim\G(a,c)$.  Note that this is not a MY type property: there is no function $f$ such that $U = f(X+Y)$ and $V=f(X) - f(X+Y)$. Further, in \cite{PW16} the converse was proved:
\begin{theorem}
Let $X$ and $Y$ be two independent positive random variables with positive densities on $(0,\infty)$ such that its logarithms are locally integrable. Let $(U, V)= T_0(X,Y)$. Suppose that $U$ and $V$ are independent. Then there exist constants $a,b,c>0$, such that $ X\sim \mathcal{K}(a,b-a,c)$, $ Y\sim \mathcal{G}(b, c)$ or, equivalently, $U\sim \mathcal{K}(b, a-b, c)$ and $V\sim\mathcal{G}(a,c)$.
\label{twPW16}
\end{theorem}
The proof was based on solving an associated functional equation. Completely different methods were used in \cite{PWkiedys}, where a regression version of this characterization was proven.  First, under integrability assumptions the recurrences for moments of $X$ and $Y$ were derived and solved. Then the integrability assumptions were  eliminated through the change of measure technique and so Theorem \ref{twPW16} holds without any assumptions on densities, even their existence. Note also that \cite{PWkiedys} contains many references on Kummer distribution including its origins, motivations and various applications. 

 In this paper we consider the HV property and the related characterization of Kummer and Gamma distributions in the cone of positive definite, symmetric matrices. An analogue of Theorem \ref{twPW16} is proven in Section \ref{sec5}. Before that, in Section \ref{sec2}, we introduce matrix--Kummer and Wishart distributions. Then, in Section \ref{sec3}, {HV property} is adapted to the matrix setting. Section \ref{sec4} is devoted to analysis of related functional equations and some technicalities. We also prove the first main result there, i.e. we solve functional equation \eqref{gl}.
  These results are applied  in Section 5 to prove the main probabilistic result, i. e. the characterization of matrix--Kummer and Wishart distributions. Possible areas of impact and open questions are presented in Section 6.

\section{The matrix Kummer distribution}\label{sec2}

Let $r\geq 1$ be an integer. Denote by $\Omega$ the linear space of real $r\times r$ symmetric matrices endowed with the inner product $\langle x,y \rangle = \tr (xy)$ for any $x, y\in\Omega$. Let  $\Omega_+\subset\Omega$ be the cone of positive-definite symmetric real $r\times r$  matrices. We denote by $\e$ the identity matrix.

For $\Sigma \in \Omega_+$ the Wishart distribution $ \W (b,\Sigma ) $ can  be defined for\break ${b\in \{0, 1/2, 1, 3/2, \ldots, (r-1)/2\}\cup \left((r-1)/2, \infty\right)}$ as the law of a random variable $Y$ valued in the closure of $\Omega_+$ with Laplace transform 

$$ \E\left( e^{\<\sigma, Y\>}\right) = 
\left(\frac{\det\Sigma}{\det(\Sigma - \sigma)}\right)^b, \;\;\mathrm{for}\;\sigma\mathrm{\;such\;that\;} \;\Sigma - \sigma \in \Omega_+.$$

If $b>\frac{r-1}{2}$, then $Y$  
has density of the form:
$$\W (b,\Sigma )(dy) = \frac{(\det \Sigma)^b}{\Gamma_r(b)}(\det y)^{b-(r+1)/2}\exp (-\langle\Sigma , y\rangle )I_{\Omega_+}(y) dy,$$
where $\Gamma_r$ is the multivariate Gamma function (see \cite{M82}) defined for any complex number $z$ with $\Re (z)> (r-1)/2$ by
$$\Gamma_r(z) = \pi^{r(r-1)/4}\prod_{j=1}^r \Gamma\left( z - \frac{j-1}{2}\right).$$
We will define matrix version of Kummer distribution following \cite{Kou12}.
We say that random variable $X$ valued in $\Omega_+$ has \textit{matrix-Kummer} distribution with parameters $a>\frac{r-1}{2}$, $b\in\mathbb{R}$, $\Sigma\in \Omega_+$, denote $X\sim \mk(a,b,\Sigma )$, if it has the following density $$\mk(a,b,\Sigma ) (dx) = C(\det x)^{a-\frac{r+1}{2}}(\det (\e+x))^{-(a+b)} \exp (-\langle\Sigma,x\rangle ) I_{\Omega_+}(x) dx,$$
where the normalizing constant $C$ equals to $\left(\Gamma _r(a)\Psi( a,\frac{r+1}{2}-b;\Sigma)\right)^{-1}$ and $\Psi$ is a confluent hypergeometric function of the second kind with matrix argument (see \cite{JJ85}, formula (2)). In the literature this distribution is sometimes called the \emph{ Kummer-gamma distribution} or the \emph{Kummer distribution of type II} (see e.g. \cite{GN09}, \cite{NC01}). It also appeared recently as a member of the family named weighted-type II Wishart distribution, \cite{ABN17}.

\section{HV property for positive definite matrices}\label{sec3}

In \cite{Kou12} Koudou showed that matrix-Kummer and Wishart distributions have the following property: if $X\sim\mk(a,b,\Sigma)$ and $Y\sim\W(b-a, \Sigma)$ are independent, then $$U = \P\left(\e+(X+Y)^{-1}\right)^{1/2} \left(\e+X^{-1}\right)^{-1} \mathrm{\; and\;} V = X+Y$$ are independent, where $\P(y)$ is endomorphism defined on $\Omega$ and for any $y\in\Omega_+$:
$$ \P(y) (x)  = yxy,\; \; x\in \Omega.$$ This is a generalization of the independence property of real-valued random variables related to transformation \eqref{KV1}. This property is in the family of Matsumoto-Yor type independence properties defined in \cite{KV11, KV12}. Recently, Ko\l odziejek showed that this property characterizes matrix--Kummer and Wishart distributions, \cite{Ko17}. In this section we establish a new independence property of Wishart and matrix-Kummer random matrices, which is  not of Matsumoto-Yor type. A related characterization is given in Section 5.

We want to find transformation that generalizes $T_0$ defined in \eqref{zero} onto $\Omega_+$ and that preserves the independence property for matrix-Kummer and Wishart distributions. 

Let $T:\Omega_+^2\to \Omega_+^2$ be defined as: 
\bel{T_def}T(x,y) =\left( \P\left[(\e+x)^{-\frac{1}{2}}\right] y,\; \P\left[\left(\e+\P\left[(\e+x)^{-\frac{1}{2}}\right] y\right)^{\frac{1}{2}}\right]x \right).\ee

Note that $T$ is involutive (as in one--dimensional case). 

 To derive the Jacobian of transformation $T$, which is done in Proposition \ref{jakob}, we need the fact that
\bel{lemLW} \Det (\P(x)) = (\det x)^{r+1},\ee  where $\Det$ is the determinant in the space of endomorphisms on $\Omega$ (see e.g. \cite{LW00} or \cite{M82}, Theorem 2.1.7).

\begin{proposition}\label{jakob}
Let $u$ and $v$ be in the cone of symmetric positive definite matrices. Denote by $T^{-1}$ the inverse of $T$ defined in \eqref{T_def}. Then the Jacobian of $T^{-1}$ is equal to
  \bel{jacobian}\j_{T^{-1}}(u,v)=(\det[\e+u])^{-(r+1)}\left( \det[\e+v+u]\right)^\frac{r-1}{2}.\ee
Moreover, since $T$ is an involution, the Jacobian of $T$ is equal to $\j_{T^{-1}}$.

\end{proposition}
The proof of Proposition \ref{jakob} is standard. The same technique was, for instance, used in \cite{MW06} for the MY property and in \cite{Kou12} for the other independence property of Wishart and Kummer matrices.
\begin{proof}
Let $x$ and $y$ be in $\Omega_+$. Then
\begin{equation}
u:= \P([\e+x]^{-\frac{1}{2}}) y\in\Omega_+,  \qquad v:=\P([\e+u]^{\frac{1}{2}})x\in\Omega_+. \label{u_v_def}
\end{equation}
 Let $T_1, T_2:\Omega_+^2\to \Omega_+^2$
by defined by
\begin{equation*}
T_1(x,y) = (w,z):= \left( x, \P([\e+x]^{-\frac{1}{2}})y\right)
\end{equation*}
and
\begin{equation*}
T_2(w,z)=(u,v):=\left( z, \P([\e+z]^\frac{1}{2}w)\right).
\end{equation*}
Then $T=T_2\circ T_1$ and we have
\bel{t1}
(x,y) = T_1^{-1}(w,z) = \left(w, \P([\e+w]^\frac{1}{2})z \right),\ee
\bel{t2}
(w,z) = T_2^{-1}(u,v) = \left(\P([\e+u]^{-\frac{1}{2}})v,u\right).
\ee

Let us note that the Jacobian $J_2$ of $T_2^{-1}$ 
equals

$$\Det \left(\begin{array}{cc}
* & E\\
\e & 0
\end{array}\right),$$

\noindent where $*$ does not need to be computed and $E$ is the differential of the function $v\mapsto [e+u]^{-\frac{1}{2}}v[e+u]^{-\frac{1}{2}}$ 

($u$ is fixed) and equals $\mathbb{P}([\e+u]^{-\frac{1}{2}})$. Hence, by \eqref{lemLW} we get
$$J_2(u,v)= \Det E = (\det[\e + u]^{-\frac{1}{2}})^{r+1} = \left( \det[\e+u]\right)^{-\frac{r+1}{2}}.$$
\noindent
The Jacobian $J_1$ of $T_1^{-1}$
can be computed in the same way: 
$$J_1(w,z) =\Det  \left(\begin{array}{cc}
\e & 0\\
* & F
\end{array}\right),
$$
where $F$ is the differential of the mapping $z\mapsto \P([\e+w]^\frac{1}{2})z$. Then
\begin{equation*}
J_1(w,z) = \Det F = \left(\det [\e+w]\right)^\frac{r+1}{2} = \det ( \e+u)^{-1} \det(\e + u + v) ,
\end{equation*}

where the last equality follows from definition of $w$ given in $\eqref{t2}$ and elementary properties of determinant. 

Finally, we obtain
\begin{equation*}J(u,v) = J_1(w,z) J_2(u,v) = \left(\det(\e + u + v) \right)^\frac{r+1}{2} \left( \det[\e+u]\right)^{-(r+1)}.
\end{equation*}
\end{proof}

\begin{theorem}\label{tw1}
 Let $X$ and $Y$ be two independent random matrices valued in $\Omega_+$. Assume that $X$ has matrix--Kummer distribution $\mk(a,b,c\e)$ and $Y$ the Wishart distribution $\W (a+b, c\e)$, where $a>\frac{r-1}{2}$, $b>\frac{r-1}{2}-a$, $c>0$. 
 
Then the random matrices $$U:=\P([\e+X]^{-\frac{1}{2}}) Y,\qquad V:=\P([\e+U]^{\frac{1}{2}})X$$
are independent. Furthermore, $U\sim \mk(a+b, -b, c\e)$ and $V\sim \W(a,c\e)$. 

\end{theorem}
\begin{proof}

Denote densities of $X$, $Y$ and $(U, V)$ by $f_X$, $f_Y$ and $f_{(U,V)}$, respectively. Since $X$ and $Y$ are independent, we have 
$$f_{(U,V)}(u,v) = |J(u,v)|f_X(x)f_Y(y) I_{\Omega_+}(u)I_{\Omega_+}(v),$$
where $(x,y) = T^{-1}(u,v) = \left(\P([\e+u]^{-\frac{1}{2}})v,  \P([\e+x]^\frac{1}{2})u\right)$ and $J$ is the Jacobian of $T^{-1}$ from Proposition \ref{jakob}. Elementary properties of trace and determinant give
\begin{eqnarray*}
\det(\e+x) &=& \det\left[ \P([\e+u]^{-\frac{1}{2}})((e+u)+v)\right]=\det(\e+u)^{-1}\det(\e+u+v),\\\label{dety}
\det(y) &=& \det\left[\P([\e+x]^\frac{1}{2})u\right] = \det (\e+x) \det u,\\\label{detx}
\det x &=& \det v \det(\e+u)^{-1},\\
\langle c\e, y\rangle &=&c\langle \e, \P(\e+x)^{1/2}u \rangle=c\langle \e, (\e+x)u \rangle,\\
\langle c\e, x+y) &=& c\langle \e, (\e+u)x+ u\rangle=c \langle \e, u\rangle +c\langle \e, \P(\e+u)^{1/2}x \rangle = \langle c\e, u\rangle +\langle c\e, v\rangle.
\end{eqnarray*}

Hence we have
\begin{eqnarray}\label{stalaC}\nonumber
f_{(U,V)}(u,v) &= &C \det(\e+u)^{-(r+1)}\det(\e+u+v)^\frac{r+1}{2}(\det x)^{a-\frac{r+1}{2}}\det(\e+x)^{-b-a}\\&&\cdot  \left(\det y\right)^{a+b-\frac{r+1}{2}}\exp\left(-c\langle \e, x+y\rangle\right)I_{\Omega_+}(u)I_{\Omega_+}(v)\\\label{gest_22}\nonumber
&=&C \det(\e+u)^{-a}\det u ^{a+b-\frac{r+1}{2}}e^{-\<c,u\>}\det v ^{a-\frac{r+1}{2}}e^{-\<c,v\>}I_{\Omega_+}(u)I_{\Omega_+}(v).
\end{eqnarray} 
\end{proof}
\begin{remark}

Constant $C$ in \eqref{stalaC} equals
$$C = \frac{c^{r(b+a)}}{\Gamma_r(a+b)}\frac{1}{\Gamma_r(a)\psi(a, \tfrac{r+1}{2}-b, c\e)}.$$
On the other hand, since $f_{(U,V)}$ is the density of $\mk(a+b, -b, c\e)\otimes \W(a,c\e)$, then
$$C = \frac{c^{ra}}{\Gamma_r(a)}\frac{1}{\Gamma_r(a+b)\psi(a+b, \tfrac{r+1}{2}+b, c\e)}.$$
So we obtain 
$$\psi(a+b, \tfrac{r+1}{2}+b, c\e)c^{rb} = \psi(a, \tfrac{r+1}{2}-b, c\e).$$
For $r=1$ it is a well known identity, see formula 13.1.29 in \cite{AS72}.

\end{remark}

Notice that $X$ and $Y$ in Theorem \ref{tw1} have very special scale parameter: the identity matrix multiplied by a positive constant $c$. We will show, in Section \ref{sec5}, that no other parameter is possible there.
\section{Functional equations}
\label{sec4}
The main result of this section is the general solution of the functional equation
 \bel{cel}A(x) + B(y) = C\left(\P (\e+x)^{-1/2}y\right) + D\left(\P[\e+\P (\e+x)^{-1/2}y]^{1/2}x\right)\ee
where $A,B,C,D: \Omega_+\to  \R$ are continuous functions.  We use techniques first developed in \cite{BW02} to solve equation of the form 
$$a(x) + b(y) = c(\P(y) x) + d(\P(y) (\e-x)),\; y\in\Omega_+, x\in \mathcal{D},$$
where $\mathcal{D} =\{z\in\Omega_+: \e-z\in\Omega_+\}$. This equation
 was concerned to prove characterization of Wishart distribution (valued in $\Omega_+$). In \cite{BW02} authors assumed that densities of considered random variables are strictly positive and twice differentiable. Earlier similar results, but under different assumptions, were obtained by Olkin and Rubin, \cite{OR62}, Casalis and Letac, \cite{CL96}, Letac and Massam, \cite{LM98}. Starting from 2013 methods from \cite{BW02} were improved by Ko\l odziejek, who:
\begin{itemize}
\item generalized Lukacs' Theorem to all non-octonion symmetric cones of rank greater than $2$ and the Lorentz cone assuming only strict positivity and continuity of densities, \cite{Ko13}, \cite{Ko16};
\item generalized independence characterization of Beta distribution to the symmetric cone setting, \cite{Ko16B}. Functional equation, which played a crucial role there was as follows
$$ a(x) + b(g(\e-x) y) = c(y) + d(g(\e-y)x),$$ where 
$x,\;y\in\mathcal{D}$, $a,b,c,d$ are continuous functions and $g$ is a division algorithm; 
\item solved the following equation
$$ a(x) + b(y) = c(x+y) + d\left(x^{-1} -(x+y)^{-1}\right)$$
for continuous $a,b,c,d$ defined on the symmetric cone, \cite{Ko15_s}. As a consequence he got a converse of  Matsumoto--Yor theorem for random variables valued in symmetric cone, i.e. for Wishart and GIG distributions. Earlier results were obtained only for the cone $\Omega_+$ and under stronger assumptions, \cite{LW00, We02a};
\item proved a new characterization of Wishart and matrix--Kummer, \cite{Ko17}.
\end{itemize}

In the proofs in this section we try to adapt the methods developed in papers cited above in order to solve \eqref{cel}. First, we recall Lemma 3.2 from \cite{Ko13}. It is formulated for any symmetric cone, but we will restrict it to our setting, i.e. the cone $\Omega_+$.
\begin{lemma}[Additive Cauchy equation]\label{jensen}
Let $f : \Omega_+ \mapsto\R$ be a measurable
function such that
$f (x) + f (y) = f (x + y)$ for all $(x, y)\in \Omega_+^2$.
Then there exists $c\in \Omega$ such that $f (x) = \<c, x\>$ for any $x\in \Omega_+$.
\end{lemma}
Next, we give solution of slightly modified logarithmic Pexider equation for functions defined on $\Omega_+ +\e := \{x\in \Omega_+:\; x-\e\in\Omega_+\}$.
\begin{proposition}\label{fog}
Let $f_1, f_2, f_3:\Omega_+ +\e\to \R$ be continuous functions such that 
\bel{fy}f_1(x) + f_2(y) = f_3(\P(x^{1/2})y) \; \mathrm{for\; all}\;\; x, y \in \Omega_+ +\e.\ee
Then there exist constants $q,$ 
$\gamma_1$, $\gamma_2\in\R$ such that for $ x\in \Omega_+ + \e$
\bel{rozne}\begin{array}{ccl}
f_1(x) &=& f_0(x)  + \gamma_1,\\
f_2(x) &=&f_0(x)+ \gamma_2,\\
f_3(x) &=&f_0(x) + \gamma_1 + \gamma_2,
\end{array}
\ee
where $f_0(x) = q\log\det x $.
\end{proposition}
\begin{proof}
Let $x = \alpha \e$, $\alpha>1$ and $\alpha\to 1^+$. Given Eq. \eqref{fy}, we have $$f_2(y) = f_3(y) - \lim_{\alpha\to 1^+} f_1(\alpha\e) = f_3(y) -\gamma_1.$$ Similarly we obtain $$f_1(x) = f_3(x) - \lim_{\alpha\to 1^+} f_1(\alpha\e) = f_3(x) - \gamma_2.$$ So Eq. \eqref{fy} is equivalent to 
\bel{ft}f(x) + f(y) = f(\P(x^{1/2})y) \; \mathrm{for\; all}\;\; x,y \in \Omega_+ +\e,\ee
where 
$f(x) = f_3(x) - \gamma_1 - \gamma_2$. 

Following the proof of Lemma 3.2 in \cite{Ko13}, we define an extension $\bar f$ of $f$ for all $x\in \Omega_+$: 
\bel{rozsz}\bar f(x) = \left\{\begin{array}{cl}
f(x), & x\in \Omega_+ + \e\\
f(t_x x) - f(t_x\e), &  x\notin \Omega_+ + \e
\end{array}\right. ,\ee
where $t_x = \tfrac{2}{\min_i \lambda_i}$, $\lambda_i$ being the $i$th eigenvalue of $x$. Also $t_{\P(x^{1/2})y}$ will be denoted by $t_{xy}$. Note that all eigenvalues of matrix $t_x x$ are greater than 1 for any $x\in \Omega_+$, so $t_x x\in \Omega_+ +\e$. Now, we will show that 
\bel{fbar}
\bar f(x) + \bar f(y) = \bar f\left(\P\left(x^{1/2}\right)y\right) \; \mathrm{for\; all}\;\; x,y \in \Omega_+.
\ee
\paragraph*{\bf Case 1:} $x \in \Omega_+ +\e$, $y \notin \Omega_++\e$ and $\P\left(x^{1/2}\right)y\in \Omega_++\e$. Then, by definition \eqref{rozsz} and Eq. \eqref{ft}
$$\bar f(x) +\bar f(y) = f(x) +f(t_y y) -f(t_y\e) =f\left(t_y \P\left(x^{1/2}\right) y\right) -f(t_y\e) = \bar f\left(\P\left(x^{1/2}\right) y\right).$$
\paragraph*{\bf Case 2:}  $x \in \Omega_+ +\e$, $y \notin \Omega_++\e$ and $\P\left(x^{1/2}\right)y\notin \Omega_+ +\e$. These imply that minimal eigenvalue of  $\P\left(x^{1/2}\right)y$ is not greater than 1. Then
\begin{multline*} \bar f(x) + \bar f (y) = f(x) + f(t_yy) - f(t_y\e) = \\ =f\left(\frac{t_y}{t_{xy}}\P\left(x^{1/2}\right)yt_{xy}\right) - f(t_y\e) = f\left( \P\left(x^{1/2}\right)y t_{xy}\right)-\left[f(t_y \e) - f\left(\frac{t_y}{t_{xy}}\e\right)\right] =\\= f\left( \P\left(x^{1/2}\right)y t_{xy}\right) - f(t_{xy}\e)=\bar f\left( \P\left(x^{1/2}\right) y\right).
\end{multline*}
Here, besides \eqref{rozsz} and \eqref{ft}, we have used the fact that every eigenvalue of $\P\left(x^{1/2}\right)y$ is not less than the product of the smallest eigenvalues of $x$ and $y$. Indeed, when $\lambda_1$ is the smallest eigenvalue of $\P\left(x^{1/2}\right)y$, then from the Min--max theorem we have 
\begin{multline*}\lambda_1 = \min_{z\in\R^n\setminus\{0\}} \frac{(x^{1/2}yx^{1/2}z,z)}{(z,z)} = \min_{z\in\R^n\setminus\{0\}} \frac{(yx^{1/2}z,x^{1/2}z)}{(x^{1/2}z,x^{1/2}z)} \frac{(x^{1/2}z, x^{1/2}z)}{(z,z)}\geq\\\geq \lambda_x\lambda_y>\lambda_y,
\end{multline*}
where $\lambda_x$ and $\lambda_y$ are the smallest eigenvalues of $x$ and $y$, respectively. The last inequality follows from the fact, that $x\in \Omega_++\e$.

Other cases can be easily verified in a similar way.

Since Eq. \eqref{fbar} holds for every $x,y\in \Omega_+$ then by Lemma 4.2 (Logarithmic Pexider Equation) from \cite{Ko15} 
$$\bar f (x) =  q\log\det x \;\; \mathrm{on} \; \Omega_+.$$
From definition $f(x) =\bar f(x)= f_0(x)$ for $x\in \Omega_++\e$ and the proof is complete.
\end{proof}
We will also need two new lemmas.
\begin{lemma}\label{lambda}
Let $c\in\Omega_+$. Assume that $\left\langle c, \P(u) v^2 \right\rangle  = \left\langle c, \P(v) u^2 \right\rangle$ for all $u,v\in \Omega_+$.  Then $c = \lambda \e$ for some $\lambda >0$.
\end{lemma}

\begin{proof}
For $v = c^{1/2}$ the equality $\left\langle c, \P(u) v^2 \right\rangle  = \left\langle c, \P(v) u^2 \right\rangle$  results in 
$$
\begin{array}{rcl}
\left\langle c, \P(u)c \right\rangle  &=& \left\langle c, \P(c^{1/2}) u^2 \right\rangle \\
\left\langle c, \P(u)c \right\rangle  &=& \left\langle c^2,  u^2 \right\rangle \\
\left\langle \e, c\cdot\P(u)c \right\rangle  &=& \left\langle \e, c^2 u^2 \right\rangle \\
0 &=&\left\langle \e, c\cdot\P(u)c- c^2 u^2 \right\rangle .
\end{array}
$$ 
On the other hand the last equality can be written as $$ 0=\left\langle \e, ucuc- u^2 c^2 \right\rangle. $$
Adding last two equalities we arrive at 
$$\begin{array}{rcl}
 0&=&\left\langle \e, ucuc +cucu- c^2u^2- u^2 c^2 \right\rangle \\
 &=&-\left\langle \e, (uc-cu)(uc-cu) \right\rangle \\
 &=&-\left\langle  (uc-cu)^\top,(uc-cu) \right\rangle .
 \end{array}
 $$
 Thus $||uc-cu||=0$ and so  $cu = uc$ for all $u\in \Omega_+$. We  conclude (see, e.g., proof of Proposition 5.2 in \cite{LW00}), that $c = \lambda \e$ for some $\lambda>0$.

\end{proof}
\begin{lemma}\label{o_x_alfa}
Let $u, z\in \Omega_+$, $\alpha>0$ and
\bel{x_def}{x_\alpha = \left[ \P(u +\frac{1}{\alpha}\e)^{-1/2} \tilde{x}\right]^2 - \e},\ee
 where $\tilde{x}=\left(\P(u +\e/\alpha)^{1/2}(z+u+\e/\alpha)\right)^{1/2}$. Then $\displaystyle\lim_{\alpha\to 0} \frac{1}{\alpha}x_\alpha$ =z.
\end{lemma}
\begin{proof}
We have
\bel{lem}
\begin{array}{ccl}
\frac{x_\alpha}{\alpha}&=&\frac{1}{\alpha}\left\{\left[ \P(u +\frac{\e}{\alpha})^{-1/2} \tilde{x}\right]^2 - \e\right\}\\
&=& \frac{1}{\alpha}\left[\P(u +\frac{\e}{\alpha})^{-1/2}\P(\tilde{x}) (u+\frac{\e}{\alpha})^{-1} -\e\right]\\
&=& \P(\alpha u+\e)^{-1/2}\left[\frac{1}{\alpha}\P(\alpha\tilde{x})(\alpha u + \e)^{-1} - u - \frac{\e}{\alpha}\right] .
\end{array}\ee
Note that 
$$(\alpha u + \e)^{-1} = \e - \left( \e +\frac{u^{-1}}{\alpha}\right)^{-1} = \e - \alpha\left( \alpha\e +u^{-1}\right)^{-1}.  $$
Indeed, $$(\alpha u + \e) \left(\e - \left( \e +(\alpha u)^{-1}\right)^{-1}\right) = \alpha u + \e  - \alpha u \left(\e + (\alpha u)^{-1}\right)\left(\e + (\alpha u )^{-1}\right)^{-1} = \e $$
and 
$$ \left(\e - \left( \e +(\alpha u)^{-1}\right)^{-1}\right)(\alpha u + \e) = \e.$$
We continue Eq. \eqref{lem}:
\nobel
\begin{array}{ccl}\nonumber
\frac{x_\alpha}{\alpha}&=&
\P(\alpha u+\e)^{-1/2}\left[ \frac{1}{\alpha}\P(\alpha\tilde{x})(\e - \alpha\left( \alpha\e +u^{-1}\right)^{-1}) - u - \frac{\e}{\alpha}\right]
\\ &=& \P(\alpha u+\e)^{-1/2}\left[ \frac{1}{\alpha}\left(\alpha^2\tilde{x}^2 - \e\right) - u - \P(\alpha \tilde{x})\left(u^{-1} + \alpha \e\right)^{-1}\right].
\end{array}\ee
Recall that $\tilde{x}=\frac{1}{\alpha}\left(\P(\alpha u + \e) (\alpha(z+u) + \e)\right)^{1/2}$. Thus
\nobel\nonumber
\begin{array}{ccl}
\frac{1}{\alpha}\left( (\alpha\tilde{x})^2 - \e\right) 
&=& \frac{1}{\alpha}\left(\P(\alpha u + \e)^{1/2}(\alpha(z+u)) + \alpha u +\e - \e \right)\\
&=&u + \P(\alpha u + \e )^{1/2} (z+u) \to 2u + z,\; \mathrm{when }\; \alpha\to 0.
\end{array}
\ee
Note that the latter calculation also implies that $\alpha\tilde{x} \to \e$ when $\alpha\to 0$.
With these observations we may eventually write that 
$$\frac{1}{\alpha}x_\alpha\to 2u + z - u - u = z\in \Omega_+ .$$
\end{proof}
The following Lemma is a simple corollary of Theorem 1 from \cite{PW16}.
\begin{lemma}\label{wym1}
Let $a$, $b$, $c$ and $d$ be continuous functions on $(0,\infty)$. Suppose that
\bel{1wym}a(x) + b(y) = c(y/(1+x)) + d(x(1+y/(1+x)))\ee
then there exist constants $a,b,c\in\R$ and $c_1+c_2 = c_3+c_4$ such that
\begin{eqnarray*}
a(x)& = &b\log x-cx-a\log(1+x) + c_1 \\
b(x)& = &a\log x - dx + c_2 \\
c(x)& = &a\log x - dx - b\log (1+x) + c_3 \\
d(x)& = &b\log x - cx + c_4
\end{eqnarray*}
\end{lemma}
In next proposition we solve matrix--variate version of Eq. \eqref{1wym}, which is our first main result. The solution will be used in the proof of the probabilistic main result of this paper -- Theorem \ref{tw2}, Section \ref{sec5}.
\begin{proposition}\label{rown}
Let $A$, $B$, $C$, $D:\Omega_+\to\R$ be continuous functions, such that 
\bel{gl}A(u) + B(v) = C\left(\P (\e+u)^{-1/2}v\right) + D\left(\P[\e+\P (\e+u)^{-1/2}v]^{1/2}u\right)\ee for any $u,v\in\Omega_+$. Then there exist constants $a,b, c_1, c_2,d\in\R$ and $ \lambda>0$ such that
\nobel
\begin{array}{rl}
A(x)=&a\log\det x - b\log\det(\e+x)+c_1 + \lambda \tr x\\
B(x)  =& b\log\det x +c_2+d+\lambda \tr x\\
C(x) =& b\log\det x - a\log\det (\e+x) +c_2+\lambda \tr x\\
D(x) = &a\log\det x +c_1+d+\lambda \tr x
\end{array}
\ee
\end{proposition}
\begin{proof}
The proof is divided into three steps.
\paragraph*{\bf Step 1}
Plugging $u=\alpha \e$, $v=\beta \e$, $\alpha,\beta>0$ into Eq. \eqref{gl}, we have
\bel{11w}
\widetilde A(\alpha ) + \widetilde B(\beta) = \widetilde C\left(\frac{\beta}{1+\alpha}\right) + \widetilde D\left(\alpha\left(1+\frac{\beta}{1+\alpha}\right)\right),
\ee
where
$$
\widetilde A(\alpha):=A(\alpha\e),\:
\widetilde B(\alpha):= B(\alpha\e), \;
\widetilde C(\alpha) :=C(\alpha\e), \;
\widetilde D(\alpha):=D(\alpha\e).$$
Since we assume that functions $A$, $B$, $C$, $D$ are continuous, then we can use Lemma \ref{wym1} and obtain, inter alia, that
\bel{a}\widetilde A(x)=a\log x - b\log(1+x) - c x + c_1,\ee
where constants $a,b,c, c_1$ are positive.
This observation will be used in Step 3.
\paragraph{\bf Step 2}
Set $v= \P(\e+\alpha \tilde u)^{1/2}x$ and $u=\alpha \tilde u$, $x, \tilde u\in\Omega_+$, $\alpha>0$, in \eqref{gl} to get:
\bel{eq_pom}
A(\alpha \tilde u) + B\left(\P(\e+\alpha \tilde u)^{1/2}x\right) = C(x) + D\left(\alpha\P(\e + x)^{1/2}\tilde u\right).
\ee
When $\alpha\to 0$ we have
$$H(\e+x):=B(x)-C(x) = \lim_{\alpha\to 0}\left\{D\left(\alpha \P(\e +x)^{1/2}\tilde u\right) - A(\alpha \tilde u) \right\}.$$
Note that the limit on the right-hand side does not depend on $\tilde u\in\Omega_+$. Therefore, for $\tilde u = \P(\e+x)^{-1/2}(\e+y)$, $y\in\Omega_+$ we get:
$$\begin{array}{ccl}
H(\e+x)&=&\displaystyle \lim_{\alpha\to 0}\left\{D\left(\alpha(\e+ y)\right) - A\left(\alpha \P(\e+x)^{-1/2}(\e+y)\right) \right\}\\
&=&\displaystyle\lim_{\alpha\to 0}\left\{D\left(\alpha(\e+ y)\right)-A(\alpha \e) +A(\alpha\e) - A\left(\alpha \P(\e+x)^{-1/2}(\e+y)\right) \right\}\\
&=&\displaystyle H(\e+y)+\lim_{\alpha\to 0}\left\{ A(\alpha\e) - A\left(\alpha \P(\e+x)^{-1/2}(\e+y)\right) \right\}

.\end{array}
$$
Denoting $$G\left(\P(\e+x)^{-1/2}(\e+y)\right) =- \lim_{\alpha\to 0}\left\{ A(\alpha\e) - A\left(\alpha \P(\e+x)^{-1/2}(\e+y)\right) \right\},$$ we have:
$$
H(\e+y) = H(\e+x) +G\left(P(\e+x)^{-1/2}(\e+y)\right) \mathrm{\;for\;any\;} x,y\in\Omega_+,
$$
which by Proposition \ref{fog} gives
  $$B(y) - C(y) = H(\e + y)=a\log\det(\e+y)+d_1$$ for any $y\in\Omega_+$, where $a,d_1\in \R$.
 
Notice that Eq. \eqref{gl} can be equivalently written as 
$$A\left(\P (\e+u)^{-1/2}v\right) + B\left(\P[\e+\P (\e+u)^{-1/2}v]^{1/2}u\right) = C(u)+ D(v).$$
Thus, if we repeat the procedure from Step 2 starting with this equation instead of \eqref{gl}, then we get $$D(y) - A(y) = b\log \det (\e+y) + d_2,\;\;b, d_2\in\R.$$ From the solution of one--dimensional Eq. \eqref{11w} it follows  that  $d_1=d_2 =d\in\R$.
  
\paragraph{\bf Step 3.}
The results of Step 2 allow us to define functions $f,g:\Omega_+\mapsto \R$  such that for $x\in\Omega_+$
\nobel
\begin{array}{rl}
A(x)=&a\log\det x - b\log\det(\e+x)+c_1 +f(x),\\
B(x)  =& b\log\det x +c_2+d+g(x),\\
C(x) =& b\log\det x - a\log\det (\e+x) +c_2+g(x),\\
D(x) = &a\log\det x +c_1+d+f(x),
\end{array}
\ee
and due to \eqref{eq_pom}, $f$ and $g$ satisfy

\bel{f_g}f(x) +g\left(\P(\e+x)^{1/2}y\right)=g(y) +f\left(\P(\e+y)^{1/2}x\right).\ee
Let $x= \alpha z_\alpha$, where $\alpha>0$,  $z_\alpha\in\Omega_+$ and $z_\alpha$ converges to $z\in \Omega_+$ when $\alpha$ tends to 0. Also set $y=y_\alpha = \beta z_\alpha ^{-1} - \e$ where $\beta >0$ is large enough for $y_\alpha$ to be in $\Omega_+$  for any $\alpha>0$ and also for the limit $\lim_{\alpha\to 0}y_\alpha\in\Omega_+$ (which is possible since $z_\alpha\to z\in \Omega_+$). Notice that
$\P(\e+y_\alpha)^{1/2} z_\alpha=\beta\e$. These observations and Eq. \eqref{f_g} allow us to write 

$$0=\lim_{\alpha\to 0}\left\{ f(\alpha z_\alpha) -f(\alpha\beta\e)\right\}.$$ 
From Step 1, Eq. \eqref{a}, we know that $\lim_{\alpha\to 0} f(\alpha\beta \e) =0$. Then  
\bel{star}\lim_{\alpha\to 0} f(\alpha z_\alpha) =0\ee
 for any $z_\alpha\in \Omega_+$ such that $z_\alpha\to z\in \Omega_+$. 

We will show that $f$ is additive. 
Firstly, we set $y= u+\e/\alpha$, $u\in\Omega_+$, and $x=x_\alpha$ defined in \eqref{x_def}. Note that $z$ used in definition  \eqref{x_def} is an arbitrary element from $\Omega_+$. 
From Lemma \ref{o_x_alfa} we know that $x_\alpha/\alpha$ converges to $z\in \Omega_+$, when $\alpha\to 0$. Thus, for $\alpha$ small enough 
$x_\alpha$ is inside the cone $\Omega_+$. Given \eqref{star}, we also have $f(x_\alpha)=f(\alpha x_\alpha/\alpha) \stackrel{\alpha\to 0}{\to} 0$. Note that $\P(\e+x_\alpha)^\frac{1}{2}(u + \e/\alpha) = z+u + \e/\alpha$. We rewrite Eq. \eqref{f_g} with those special $x$ and $y$. After taking the limit as  $\alpha\to0$ we obtain 
\bel{final1}f(z) = \lim_{\alpha\to 0}\left\{g\left(z+u+\frac{1}{\alpha}\e\right) - g\left(u+\frac{1}{\alpha}\e\right) \right\} \ee

On the other hand, if we plug $x=\alpha u$ and $y = \e/\alpha$  in Eq. \eqref{f_g} and take the limit as $\alpha\to 0$, we obtain
$\lim_{\alpha\to 0} \left\{g(u+\e/\alpha  ) - g(\e/\alpha)\right\} = f(u)$. Combining this result with Eq. \eqref{final1} we have
$$\begin{array}{ccl}
f(u) + f(z) &=& \lim_{\alpha\to 0} \left\{g\left(u+\frac{1}{\alpha}\e \right) - g\left(\frac{1}{\alpha}\e\right) + g\left(z+u+\frac{1}{\alpha}\e\right) - g\left(u+\frac{1}{\alpha}\e\right) \right\}\\
&=& \lim_{\alpha\to 0} \left\{ g\left(z+u+\frac{1}{\alpha}\e\right)- g\left(\frac{1}{\alpha}\e\right)   \right\}\\
&=& f(z+u)
\end{array}
$$ 
Note that this equation, $f(u) + f(z) = f(u+z)$, holds for all $u,z\in\Omega_+$. By Lemma \ref{jensen} we conclude that $f(x) = \left\langle c, x\right\rangle$ where $c\in \Omega_+$. Similarly, due to symmetry in \eqref{f_g}, we show that $g(x) = \left\langle \tilde{c}, x\right\rangle$. Eq. \eqref{f_g} with $x =\alpha u$, $y=\e/\alpha$ implies $\tilde c = c$. 

The last step of the proof is to show that $c=\lambda\e$ for real and positive $\lambda$. We will use Lemma \ref{lambda} to do that. For $f$ and $g$ identified above, Eq. \eqref{f_g} assumes the form
$$  \left\langle c, \alpha^2\P(\e+x)^{1/2}(\e+y)\right\rangle=  \left\langle c,\alpha^2\P(\e+y)^{1/2}(\e+x)\right\rangle$$
for any $\alpha>0$. Note that for any $u,v\in \Omega_+$ there exist $\alpha>0$, $x, y\in \Omega_+$ such that $\alpha(\e+x) = u^2$ and $\alpha(\e+y) = v^2$. Thus, we can use Lemma \ref{lambda} to conclude that  $c=\lambda \e$, where $\lambda>0$. 
Consequently, we have
\nobel
\begin{array}{rl}
A(x)=&a\log\det x - b\log\det(\e+x)+c_1 + \lambda\tr x,\\
B(x)  =& b\log\det x +c_2+d+\lambda\tr x,\\
C(x) =& b\log\det x - a\log\det (\e+x) +c_2+\lambda\tr x,\\
D(x) = &a\log\det x +c_1+d+\lambda\tr x.
\end{array}
\ee
\end{proof}

\section{Characterization of matrix--Kummer and Wishart distributions}\label{sec5}
In this section we prove the converse to the independence property from Theorem \ref{tw1}, that is  a new characterization of the martix--Kummer and the Wishart distributions. Similarly to the one--dimensional case considered in \cite{PW16}, we need to impose some regularity conditions on densities.
\begin{theorem}\label{tw2}
Let $X$ and $Y$ be independent random variables valued in $\Omega_+$ with positive and continuous densities. Assume that random matrices 
$$U = \P [(\e+X)^{-1/2}]Y\;\mathit{and}\; V = \P [ ( \e+ U)^{1/2}]X$$ are also independent. 

Then there exist $a>(r-1)/2$, $b>(r-1)/2 - a$ and $\lambda>0$ such that $X\sim \mk (a, b, \lambda\e)$ and $Y\sim\W(a+b, \lambda\e)$.
\end{theorem}
\begin{proof}
Recall that $$T(x,y) =\left( \P\left[(\e+x)^{-\frac{1}{2}}\right] y,\; \P\left[\left(\e+\P\left[(\e+x)^{-\frac{1}{2}}\right] y\right)^{\frac{1}{2}}\right]x \right).$$ Then $(U,V) = T(X,Y)$ and $(X,Y) = T(U,V)$.

Independence of random variables together with continuity of their densities imply
\bel{eq20}f_U(u)f_V(v) = |J(u,v)|f_X(x)f_Y(y)
 \;\;\mathrm{for\; all\;} u, v\in \Omega_+,\ee
where $(x,y) =T(u,v)$ and the Jacobian $J$ of $T^{-1}$ is given in 
{Proposition \ref{jakob}}.
Taking logarithms of both sides in \eqref{eq20} and defining functions $A,B,C,D:{\Omega_+\to \R}$ as
$$\begin{array}{l}
A(u) =\log f_U(u)+\frac{r+1}{2}\log\det u, \\
B(u) =\log f_V(u)+\frac{r+1}{2}\log\det u,\\
C(u) =\log f_X(u) +\frac{r+1}{2}\log\det u, \\
D(u) =\log f_Y(u) +\frac{r+1}{2}\log\det u, 
\end{array}
$$
we can rewrite Eq. \eqref{eq20} in the following way
\bel{eq21}
A(u) + B(v) = C\left(\P (\e+u)^{-1/2}v\right) + D\left(\P\left(\e+\P (\e+u)^{-1/2}v\right)^{1/2}u\right),\;\; u,v\in\Omega_+.
\ee

From Proposition \ref{rown} it follows
\nobel
\begin{array}{rl}
A(x)=&a\log\det x - b\log\det(\e+x)+c_1 + \lambda\tr x,\\
B(x)  =& b\log\det x +c_2+d+ \lambda\tr x,\\
C(x) =& b\log\det x - a\log\det (\e+x) +c_2+ \lambda\tr x,\\
D(x) = &a\log\det x +c_1+d+ \lambda\tr x.
\end{array}
\ee
The latter and the fact, that functions $A$, $B$, $C$ and $D$ represent logarithms of densities of random variables, imply  $X\sim \mk (a, b, \lambda\e)$ and $Y\sim\W(a+b, \lambda\e)$.

\end{proof}
\section{Concluding remarks}
Recently P. Vallois indicated \footnote{During his talk at 10th International Conference of the ERCIM WG on Computational and Methodological Statistics (CMStatistics 2017),  16-18.12.2017.} that one can define a transformation which generalizes $T_0$ for random matrices, is different from \eqref{T_def} and also preserves independence of Wishart and matrix-Kummer random matrices. Namely, let 
$$T(x,y) = \left(\P(\e+x+y)(\e+x)^{-1} -\e,\; x+y - \left[\P(\e+x+y)(\e+x)^{-1} -\e\right]\right).$$
Vallois says, and this can be checked in standard way, that if $X\sim\mk(a,b,\Sigma)$ and $Y\sim \W(a+b, \Sigma)$ are independent, then $(U,V)=T(X,Y)$ are also independent. Note that since  $U+V = X+Y$ here, then $\Sigma$ can be any positive definite matrix, which was not true in our case. If the converse theorem holds, remains an open question.

We hope that the probabilistic results of this paper can help to state and prove an  analogous property in free (non-commutative) probability. Let us recall that in the case of the Matsumoto-Yor property, its analogue in free probability was accomplished through an appropriate matrix independence property, \cite{KSz17}. This problem is currently being under study. 

In \cite{PW16} authors formulated multivariate characterization of a product of $p-1$ Kummer random variables and one Gamma random variable, $p\geq 2$. There, Kummer is a marginal distribution of a certain $p$-dimensional distribution, called tree-Kummer distribution in the paper, see Section 3 in \cite{PW16}. For instance, when $p=2$, then this density is of the form
$$
f(x_1, x_2) \propto x_1^{a_1-1}x_2^{a_2 -1}\exp\{ -c(x_1+x_2 + x_1x_2)I_{(0,\infty)^2}(x_1,x_2) \}.
$$
Similarly, matrix-Kummer distribution appears naturally as a marginal distribution of the following generalization of bi-Wishart distribution
$$
f_{(X,Y)}(x,y) \propto (\det x)^{p-\tfrac{r+1}{2}} (\det y)^{q-\tfrac{r+1}{2}} \exp(-\langle c,x+y+xy\rangle)I_{\Omega_+\times\Omega_+}(x,y),\; c>0.
$$
Then the conditional distribution of $X$ given $Y$ is Wishart $\W(p, c(\e+y))$, while its marginal distribution is matrix-Kummer $\mk(p, q-p, c)$. Also a question arises if a multivariate version of our independence characterization holds in a matrix setting? 

\section{Acknowledgments}
It is a pleasure to thank J. Weso\l owski for his invaluable help and comments which greatly improved the manuscript. 

This research was supported by the grant 2016/21/B/ST1/00005 of National Science Center, Poland.



\end{document}